\documentclass[11pt]{amsart}

\usepackage[english]{babel}
\usepackage[utf8]{inputenc}
\usepackage{amsmath}
\usepackage{graphicx}
\usepackage{amssymb}
\usepackage{amsthm}
\usepackage{tikz-cd}
\usepackage{mathrsfs}
\usepackage[colorinlistoftodos]{todonotes}
\usepackage{enumitem}
\usepackage{yfonts}
\usepackage{ dsfont }

\newtheorem{thm}{Theorem}[section]
\newtheorem{lem}[thm]{Lemma}
\newtheorem{prop}[thm]{Proposition}

\begin{document}

\title{Homotopies of constant Cuntz class}

\author[AT]{Andrew S. Toms}

\address{Department of Mathematics, Purdue University, 150 N University St, West Lafayette IN 47907, USA \ {\tt atoms@purdue.edu}}

\maketitle

\begin{abstract}

Let $A$ be a unital simple separable exact \mbox{C$^*$-algebra} which is approximately divisible and of real rank zero.  We prove that the set of positive elements in $A$ with a fixed Cuntz class is path connected.  This result applies in particular to irrational rotation algebras and AF algebras.  \end{abstract}

\section{Introduction}\label{intro}

The Cuntz semigroup of a C$^*$-algebra was introduced by J. Cuntz in 1978 with a view to examining rank functions on C$^*$-algebras \cite{Cuntz1}.  It has enjoyed renewed relevance in recent years owing both to its role in the classification theory of simple separable amenable C$^*$-algebras and to its capacity for reflecting C$^*$-algebraic properties in a purely algebraic setting.  A salient example is the recent progress on the Toms-Winter conjecture and its confirmation in both the unique trace and stable rank one cases.  This has been a concerted effort by many authors and an exhaustive list of credits would be prohibitively long, but some important works include \cite{APRT}, \cite{CETWW}, \cite{ET}, and \cite{W1}.  It has firmly established the breadth and importance of C$^*$-algebras which absorb the Jiang-Su algebra $\mathcal{Z}$ tensorially, a class that includes the approximately divisible C$^*$-algebras considered here (\cite{BKR}, \cite{TW2}).

A steady theme in applications of the Cuntz semigroup is that what is true of projections and Murray-von Neumann equivalence is often true also of positive elements and Cuntz equivalence, albeit with greater effort.  We continue this theme with an investigation of the homotopy properties of positive elements in a class of C$^*$-algebras which includes, notably, the irrational rotation algebras.
When we say that two projections are homotopic, it is implicit that the homotopy consists of projections whose Murray-von Neumann equivalence classes are necessarily constant.  Na\"ively, homotopies of positive elements are less interesting:  any two positive elements of a C$^*$-algebra are homotopic via the line segment between them, and this segment consists of positive elements. 
A better question, however, one that extends the spirit of homotopy for projections, is to ask whether two positive elements are homotopic via a path of positive elements whose Cuntz classes are constant.  It is this question that we address here.  
The chief difficulty is that unlike Murray-von Neumann classes, Cuntz classes are not robust under small perturbations in norm.  Indeed, every positive element in a unital C$^*$-algebra is arbitrarily close in norm to a representative of the Cuntz class of the unit.  Constructing homotopies of positive elements with constant Cuntz class is therefore a delicate matter.

\begin{thm}\label{cuntzhom}
Let $A$ be a unital simple separable exact C$^*$-algebra which is approximately divisible and of real rank zero.  For each $a \in A_+$, let $\langle a \rangle$ denote its Cuntz class. It follows that the set 
\[
S_a = \{ b \in A_+ \ | \ \langle b \rangle=\langle a \rangle \}
\]
is path connected.
\end{thm}

\noindent
$\mathrm{K}$-theory is rooted in the classification of vector bundles over a topological space $X$ up to homotopy. Assuming that $X$ is compact and Hausdorff, these can be thought of as homotopy classes of projections in $\mathrm{M}_n(\mathrm{C}(X))$ for large enough $n$.  Positive elements in this algebra with the same Cuntz class must at least have the same rank at every point of $X$, and so homotopies as in Theorem \ref{cuntzhom} would have to preserve rank pointwise.  Such homotopies are obstructed by the $\mathrm{K}$-theory of closed subsets of $X$.  Indeed, this is the fundamental observation of \cite{To1}.  Our result says that for positive elements as in Theorem \ref{cuntzhom} which are not homotopic to projections, this obstruction vanishes---these elements behave like trivial vector bundles, with their homotopy class determined only by their pointwise rank with respect to lower semicontinuous dimension functions.
We expect that the conditions of real rank zero and approximate divisibility can be relaxed to the property of $\mathcal{Z}$-stability.

The paper is organized as follows:  Section \ref{prelim} recalls briefly some facts and tools involving the Cuntz semigroup and the notion of approximate divisibility;  Section \ref{disc} reduces the proof of Theorem \ref{cuntzhom} to the special case of certain positive elements with completely disconnected spectrum;  and Section \ref{main} appeals to a result of R\o rdam on $\mathrm{K}_1$-surjectivity to prove Theorem \ref{cuntzhom} for the restricted class of elements found in Section \ref{disc}.

\vspace{2mm}
\noindent
{\bf Acknowledgement.}. The author would like to thank N. C. Phillips for several helpful conversations and A. Asadi-Vasfi for a careful reading of the manuscript.


\section{Preliminaries}\label{prelim}

A full treatment of the basic theory of the Cuntz semigroup can be found in \cite{APT}.  We make a brief recollection here  of the details required for the sequel.  Let $A$ be a $C^*$-algebra and let $A_+$ denote the subset of positive elements.  If $a,b \in A_+$, we write $a \precsim b$ if there is a sequence $(v_n)$ in $A$ such that $v_n b v_n^* \to a$.  If $a \precsim b$ and $b \precsim a$ then we write $a \sim b$ and say that $a$ and $b$ are Cuntz equivalent.  Set $\mathrm{Cu}(A) =( A \otimes \mathcal{K}_+) / \sim$, and let $\langle a \rangle$ denote the equivalence class of $a$ in $\mathrm{Cu}(A)$.  Now $\langle a \rangle+\langle b \rangle=\langle a \oplus b \rangle$ defines addition on $\mathrm{Cu}(A)$ and $\langle a \rangle \leq \langle b \rangle \Leftrightarrow a \precsim b$ defines a partial order.
These make $\mathrm{Cu}(A)$ into an ordered semigroup, the Cuntz semigroup.

It was shown in \cite{CEI} that $\mathrm{Cu}(A)$ admits suprema for increasing sequences.  We say that $\langle a \rangle$ is compactly contained in $ \langle b \rangle$, written $\langle a \rangle \ll \langle b \rangle$, if whenever $(\langle b_i \rangle)$ is increasing and $\sup \langle b_i \rangle = \langle b \rangle$ we have $\langle a \rangle \leq \langle b_{i_0}\rangle$ for some $i_0 \in \mathbb{N}$.  An element $\langle a \rangle$ of $\mathrm{Cu}(A)$ is called compact if $\langle a \rangle \ll \langle a \rangle$.  If $A$ is simple or of stable rank one, then $\langle a \langle$ is compact if and only if $a \sim p$ for some projection $p$ (\cite{APT}, \cite{BPT}).  Recall that $A$ has real rank zero if every self-adjoint element in $A$ can be approximated in norm by self-adjoint elements with finite spectrum.  Consequently, for such algebras, each $\langle a \rangle \in \mathrm{Cu}(A)$ is seen to be the supremum of a sequence of compact elements by exploiting (i) below.

Let $\delta>0$, and define a continuous map $g_\epsilon: \mathbb{R} \to [0,\infty)$ by 
\[
g_\epsilon(t) = \max \{ 0,t-\epsilon \}.
\]
Set $(a-\epsilon)_+ = g_\epsilon(a)$.  We record three useful facts related to this construction (see \cite{APT}):
\begin{enumerate}
\item[(i)] $a \precsim b$ if and only if $(a-\epsilon)_+ \precsim b$, for every $\epsilon >0$;
\item[(ii)] $\langle (a-\epsilon)_+ \rangle \ll \langle a \rangle$ for every $\epsilon >0$;
\item[(iiii)] If $\|a-b\| < \epsilon$, then $(a-\epsilon)_+ \precsim b$ for every $\epsilon >0$.
\end{enumerate}
The next Lemma is Proposition 2.7 (i) of \cite{KR}.
\begin{lem}\label{herrep}
Let $A$ be a C$^*$-algebra with $a,b \in A_+$.  If $b \in \overline{aAa}$, then $b \precsim a$.
\end{lem}
\noindent
We record one more technical Lemma for use in the sequel.  
\begin{lem}\label{projrep}
Let $A$ be a C$^*$-algebra and let $0 \neq a \in A_+$.  Suppose that zero is an isolated point of the spectrum of $a$.  It follows that there are a homotopy $h:[0,1] \to \mathrm{C}^*(a)$ and a projection $p \in \mathrm{C}^*(a)$ such that $h(0) = a$, 
$h(1)=p$, and $\langle h(t) \rangle= \langle a \rangle$ for every $t \in [0,1]$.
\end{lem}
\begin{proof}
Let $\epsilon >0$ be the smallest nonzero element of the spectrum $\sigma(a)$.  Define $h(t) = a^{\frac{1}{1-t}}$ for $t \in [0,1)$ and $p=h(1) = \chi_{[\epsilon,\infty)}(a)$.  This is clearly a homotopy inside $\mathrm{C}^*(a)$, and $\langle h(t) \rangle= \langle a \rangle$ by Proposition 2.5 of \cite{APT}.
\end{proof}

Let $A$ be unital and exact, and let $\mathrm{T}(A)$ denote the space of tracial states on $A$.  For $\tau \in \mathrm{T}(A)$ and $a \in (A \otimes \mathcal{K})_+$ we define 
\[
d_\tau(a) = \lim_{n \to \infty} \tau(a^{1/n}).
\]
This defines a lower semicontinuous dimension function on $A$.  We say that $A$ has strict comparison of positive elements if $a \precsim b$ whenever
\[
d_\tau(a) < d_\tau(b), \ \forall \tau \in \mathrm{T}(A).
\]

A unital separable C$^*$-algebra $A$ is said to be approximately divisible if there exists, for each $N \in \mathbb{N}$, an approximately central sequence of unital $*$-homomorphisms $\phi_n: \mathrm{M}_N(\mathbb{C}) \oplus \mathrm{M}_{N+1}(\mathbb{C}) \to A$ (\cite{BKR}).  It is shown in \cite{BKR} that the irrational rotation algebras are approximately divisible.  Approximately divisible C$^*$-algebras are $\mathcal{Z}$-stable, and simple $\mathcal{Z}$-stable C$^*$-algebras have strict comparison of positive elements (\cite{Ro1}, \cite{TW2}). 


\section{Totally disconnected spectra}\label{disc}

In this section we reduce the proof of our main result to the case of positive elements constructed from orthogonal projections.
While the statement here is different, the essential idea for the following Lemma is contained in Lemma 2.19 of \cite{APT}.

\begin{lem}\label{split}
Let $A$ be a unital simple separable exact C$^*$-algebra with strict comparison of positive elements.  Suppose further that $A$ is stably finite.  Let $a \in A$ be positive, and let $q \in A \otimes \mathcal{K}$ be a projection such that $q \precsim a$.  It follows that there are a projection $p \sim q$ and a positive element $b$ in $\overline{aAa}$ such that $bp = pb =0$ and
\[
d_\tau(a) = d_\tau(b) + d_\tau(p), \ \forall \tau \in \mathrm{T}(A).
\]
\end{lem}

\begin{proof}
Let $0< \epsilon <1/2$ be given and find $\delta >0$ such that
\[
(q-\epsilon)_+ \precsim (a-\delta)_+.
\]
It follows that for any $0< \gamma < 1/2$ there is $v \in A \otimes \mathcal{K}$ such that
\[
v(a-\delta)_+v^* = (q - \epsilon -\gamma)_+ \sim q.
\]
Set
\[
s = (a-\delta)_+^{1/2} v^*v (a-\delta)_+^{1/2},
\]
so that $q \sim s$ and $\langle s \rangle$ is compact in $\mathrm{Cu}(A)$ (\cite{APT}, Corollary 2.6).  There is then a projection $p$ in C$^*(s)$ with $s \sim p$ by Lemma \ref{projrep}.  Clearly,
\[
p \in \overline{(a-\delta)_+^{1/2} A (a-\delta)_+^{1/2}}.
\]
Set $f_\delta(t) = \min \{ t/\delta, 1\}$, $t \geq 0$, so that $f_\delta(a)p = pf_\delta(a) =p$.  Note that \mbox{$a \sim f_\delta(a)$} as these elements both have the same support in C$^*(a) = \mathrm{C}(\sigma(a))$.  Now, 
\begin{eqnarray*}
f_\delta(a) & = &  f_\delta(a)(p+(1_A-p)) \\
& = & f_\delta(a)p + f_\delta(a)(1_A-p) \\
& = & p + (1_A-p)f_\delta(a) (1_A-p) \\
& = & p \oplus (f_\delta(a) - p)
\end{eqnarray*}
Setting $b =  f_\delta(a) - p$ and observing that $a \sim f_\delta(a)$, we have
\[
d_\tau(a) = d_\tau(f_\delta(a)) = d_\tau(p \oplus b) = d_\tau(p) + d_\tau(b),
\]
as desired.
\end{proof}

\begin{prop}\label{projsum}
Let $A$ be a unital simple separable exact C$^*$-algebra which is approximately divisible and of real rank zero.  Suppose further that $A$ is stably finite, and let $a \in A $ be positive.  It follows that there is a sequence $p_1,p_2,\ldots$ of mutually orthogonal projections in $\overline{aAa}$ such that with
\[
b = \bigoplus_{i=1}^\infty \frac{1}{2^i} p_i,
\]
we have a homotopy $h:[0,1] \to (\overline{aAa})_+$ satisfying $h(0) = a$, $h(1) = b$, and $\langle h(t) \rangle = \langle a \rangle \in \mathrm{Cu}(A)$ for each $t \in [0,1]$.
\end{prop}

\begin{proof}
If $\langle a \rangle \in \mathrm{Cu}(A)$ is compact, then zero is an isolated point in the spectrum of $a$ and the Proposition follows from Lemma \ref{projrep}.

If $a$ is not compact, then real rank zero implies the existence of a sequence $q_1,q_2,\ldots$ of projections in $A \otimes \mathcal{K}$ such that $q_i \neq q_{i+1}$, $\langle q_i \rangle \ll \langle q_{i+1} \rangle$, and 
\[
\sup_i \langle q_i \rangle = \langle a \rangle
\]
\noindent
Since $A$ is approximately divisible, the $\mathrm{K}_0$-group of $A$ is weakly unperforated.  It follows that $\langle q_{i} \rangle - \langle q_{i-1} \rangle  = \langle r_i \rangle$ for each $i >1$ and some projection $0 \neq r_i \in A \otimes \mathcal{K}$.  Set $r_1 = q_1$.  Apply Lemma \ref{split} with $q = r_1$ and $a$ as in the present Proposition to find a projection $p_1 \sim r_1$ and positive element $b_1$ in $\overline{aAa}$ such that $p_1 b_1 = 0$ and 
\[
d_\tau(a) = d_\tau(p_1) + d_\tau(b_1), \ \forall \tau \in \mathrm{T}(A).
\]
Assume, inductively, that we have found in $\overline{aAa}$ mutually orthogonal projections $p_i \sim r_i$, $1 \leq i \leq n$,  and a positive element $b_n$ orthogonal to each $p_i$ such that
\begin{equation}\label{induct}
d_\tau(a) = d_\tau \left( \oplus_{i=1}^n p_i \right) + d_\tau(b_n) = \sum_{i=1}^n  d_\tau(p_i) + d_\tau(b_n).
\end{equation}
Applying Lemma \ref{split} with $a=b_n$ and $q=r_{n+1}$ and labeling the resulting projection and positive element as $p_{n+1}$ and $b_{n+1}$, respectively, yields (\ref{induct}) above with $n$ replaced by $n+1$.  Continuing inductively in this manner produces a sequence of mutually orthogonal projections $p_1,p_2,\ldots$ in $\overline{aAa}$.

Set 
\[
b = \bigoplus_{i=1}^\infty \frac{1}{2^i} p_i.
\]
Note that zero is an accumulation point of the spectrum of $b$, so that $b$ is not compact in $\mathrm{Cu}(A)$.  Now
\[
q_n \sim \oplus_{i=1}^n p_i \precsim b \precsim a
\]
by Lemma \ref{herrep} so that 
\[
d_\tau(q_n) \leq d_\tau(b) \leq d_\tau(a), \ \forall \tau \in \mathrm{T}(A).
\]
Since $\langle a \rangle = \sup_n \langle q_n \rangle$ in $\mathrm{Cu}(A)$ we then have
\[
d_\tau(a) = \sup_n d_\tau(q_n) \leq d_\tau(b) \leq d_\tau(a), \ \forall \tau \in \mathrm{T}(A)
\]
and so 
\[ 
d_\tau(b) = d_\tau(a), \  \forall \tau \in \mathrm{T}(A).
\]
$A$ is approximately divisible, hence $\mathcal{Z}$-stable, and so has strict comparison of positive elements (\cite{Ro1}, \cite{TW2}).  Since neither $a$ nor $b$ is compact it follows that $\langle a \rangle= \langle b \rangle$ in $\mathrm{Cu}(A)$ (\cite{BPT}). 

Set 
\[
h(t) = (1-t)a +tb, \ t \in [0,1].
\]
Evidently, $\langle h(0) \rangle = \langle h(1) \rangle =  \langle a \rangle$.   Since $h(t) \in \overline{aAa}$ we have $\langle h(t) \rangle \leq \langle a \rangle$.  On the other hand, for $t \in (0,1)$,
\[
a \sim (1-t)a \precsim (1-t)a + tb = h(t),
\]
whence $ \langle h(t) \rangle \geq \langle a \rangle$ also.  
\end{proof}

\section{Path connected Cuntz classes}\label{main}

The following Lemma is due to R\o rdam (\cite{Ro1}, Lemma 6.3), although we have expanded its statement slightly.  Recall that a unital C$^*$-algebra $A$ is strongly $\mathrm{K}_1$-surjective if the canonical map 
\[
\mathcal{U}(B + \mathbb{C}(1_A)) \to \mathrm{K}_1(A)
\]
is surjective for every full hereditary subalgebra $B$ of $A$.

\begin{lem}\label{K1rep}
Let $A$ be a unital approximately divisible C$^*$-algebra.  Then $A$ is strongly $\mathrm{K}_1$-surjective.  In particular, for any full projection $p \in A$ and $g \in \mathrm{K}_1(A)$, there is a unitary $v \in pAp$ such that $g = [v+ (1_A-p)]_1$ in $\mathrm{K}_1(A)$.
\end{lem}

\begin{lem}\label{projhom}
Let $A$ be a unital simple approximately divisible C$^*$-algebra and let $p, q$ be full projections in $A$ with $[p] =[q]$ in $\mathrm{K}_0(A)$.  It follows that there is a continuous path of unitaries $(u_t)_{t \in [0,1]}$ in $A$ such that $u_0 = 1_A$ and $u_1 p u_1^* = q$.  In particular, $p$ and $q$ are homotopic via projections in $A$.
\end{lem}

\begin{proof}
Since $[p] = [q]$ in $\mathrm{K}_0(A)$, we have also that $[1_A - p] = [1_A - q]$.  The approximate divisibility of $A$ implies that $A$ has stable rank one, hence enjoys cancellation of projections.  Thus, $p \sim q$ and $1_A-p \sim 1_A-q$ in the Murray-von Neumann semigroup.  It follows that there is a unitary $v \in \mathcal{U}(A)$ such that $vpv^* = q$.  We now modify $v$ to ensure that it is in the connected component of $1_A \in \mathcal{U}(A)$.  

If $p=1_A$ there is nothing to prove, so we may assume that $1_A-p$ is full.  Set $g = [v]_1$ in $\mathrm{K}_1(A)$.  Apply Lemma \ref{K1rep} with $1_A-p$ in place of $p$ to find a unitary $w \in \mathcal{U}((1_A-p)A(1_A-p) + p)$ such that $[w]_1 = -g$ in $\mathrm{K}_1(A)$.  Now $[vw]_1 = 0$ in $\mathrm{K}_1(A)$ and $A$ has stable rank one, so there is a continuous path $(u_t)_{t \in [0,1]}$ in $\mathcal{U}(A)$ such that $u_0=1_A$ and $u_1 = vw$.  Observing that
\[ 
u_1 p u_1^* = (vw) p (vw)^* = vwpw^*v^* = vpv^* = q
\]
completes the proof.
\end{proof}

We are now ready to prove our main result.

\begin{proof} 
(Theorem \ref{cuntzhom})
Let $A$ be a simple unital separable exact C$^*$-algebra which is approximately divisible and of real rank zero, and let $a,b \in A$ be positive with $\langle a \rangle=\langle b \rangle$ in $\mathrm{Cu}(A)$.  
Since $A$ is $\mathcal{Z}$-stable by Theorem 2.3 of \cite{TW2}, it is either stably finite (when it has a trace) or purely infinite (when it is traceless) (\cite{Ro1}).  In the purely infinite case, the Cuntz classes of any two nonzero positive elements are the same, so the line segment $(1-t)a+tb, \ t \in [0,1]$, defines a homotopy between $a$ and $b$. This segment consists entirely of nonzero positive elements and the Theorem follows.  We may thus assume that $A$ is stably finite, whence $\mathrm{T}(A) \neq \emptyset$.  In fact, $A$ has stable rank one by Theorem 6.7 of \cite{Ro1}.

If $\langle a \rangle=\langle b \rangle$ is compact in $\mathrm{Cu}(A)$ then zero is not an accumulation point of the spectrum of $a$ or $b$ by Proposition 2.23 of \cite{APT}.  By Lemma \ref{projrep} we may assume that $a$ and $b$ are projections and the desired result follows from Proposition \ref{projhom}.  

It remains to address the case where $\langle a \rangle= \langle b \rangle$ is not compact, so that $a$ and $b$ have zero as an accumulation point of their spectrum.  By Proposition \ref{projsum}, we may assume that there are sequences of projections $p_1,p_2,\ldots$ and $q_1,q_2,\ldots$ in $\overline{aAa}$ and $\overline{bAb}$, respectively, with the following properties:
\begin{enumerate}
\item[(i)] $p_ip_j = 0 = q_iq_j$ for every pair $i \neq j$;
\item[(ii)] $p_i \sim q_i$ for each $i \in \mathbb{N}$;
\item[(iii)] 
\[
a = \bigoplus_{i=1}^\infty \frac{1}{2^i} p_i \ \mathrm{and} \ b = \bigoplus_{i=1}^\infty \frac{1}{2^i} q_i.
\]
\end{enumerate}
For convenience we set $s_k = \oplus_{i=1}^k 1/2^i$.  Assume that we have found a continuous path $u_t:[0,s_k] \to \mathcal{U}(A)$ with the property that $u_0 = 1_A$ and
\begin{equation}\label{conjequal}
u_{s_k} \left( \oplus_{i=1}^k p_i \right) u_{s_k}^* = \oplus_{i=1}^k q_i.
\end{equation}
(The existence of this path for $k=1$ follows from (ii) above and Lemma \ref{projhom}.)  Let us show that $u_t$ can be extended to $[0,s_{k+1}]$ so that (\ref{conjequal}) holds with $k+1$ in place of $k$ and moreover that
\begin{equation}\label{fix}
u_t \left( \oplus_{i=1}^k p_i \right) u_t^* = \oplus_{i=1}^k q_i, \ \forall t \in [s_k,s_{k+1}].
\end{equation}

Set 
\[
Q_k = 1_A - \left( \oplus_{i=1}^k q_i \right).
\]
Note that $Q_k A Q_k$ is approximately divisible by \cite{BKR}[Corollary 2.9], and that $u_{s_k} p_{k+1} u_{s_k}^* \sim q_{k+1}$ in $Q_k A Q_k$.  It follows from Lemma \ref{projhom} that there is a continuous path 
\[
\tilde{v}_t:[s_k,s_{k+1}] \to \mathcal{U}(Q_k A Q_k)
\]
such that $\tilde{v}_{s_k} = Q_k$ and 
\[
\tilde{v}_{s_{k+1}} (u_{s_k} p_{k+1} u_{s_k}^*) \tilde{v}_{s_{k+1}}^* = q_{k+1}.
\]
Set $v_t = \tilde{v}_t + (1_A-Q_k), \ t \in [s_k,s_{k+1}]$, so that $v_t \in \mathcal{U}(A)$ and $v_{s_k} = 1_A$.  We now define our extension of $u_t:[0,s_{k+1}] \to \mathcal{U}(A)$ via
\begin{equation}\label{extend}
u_t = \left\{ \begin{array}{rcl} u_t & = & t \in [0,s_k) \\ v_t u_{s_k} & = & t \in [s_k,s_{k+1}] \end{array} \right. .
\end{equation}
This extension is continuous since $v_{s_k} = 1_A$.  Also,
\begin{align}
u_{s_{k+1}} \left( \oplus_{i=1}^{k+1} p_i \right) u_{s_{k+1}}^* & =  v_{s_{k+1}} \left( u_{s_{k}} \left( \oplus_{i=1}^{k+1} p_i \right) u_{s_{k}}^* \right) v_{s_{k+1}}^* \\
& =  v_{s_{k+1}} \left( u_{s_{k}} \left( \oplus_{i=1}^{k} p_i \right) u_{s_{k}}^* \right) v_{s_{k+1}}^* \\
&    \hspace{10mm} +  \ v_{s_{k+1}} u_{s_k} p_{i+1} u_{s_{k}}^*  v_{s_{k+1}}^*  \\
& =  v_{s_{k+1}} \left( \oplus_{i=1}^k q_i \right) v_{s_{k+1}}^* \\
& \hspace{10mm} + \ \ \tilde{v}_{s_{k+1}} u_{s_k} p_{i+1} u_{s_{k}}^*  \tilde{v}_{s_{k+1}}^* \\
&= \left( \oplus_{i=1}^k q_i \right) + q_{k+1} \\
&= \oplus_{i=1}^{k+1} q_i
\end{align}
satisfying (\ref{conjequal}) with $k+1$ in place of $k$, as required.  Note also that for $t \in [s_k,s_{k+1}]$ we have
\begin{align}
u_t \left( \oplus_{i=1}^k p_i \right) u_t^* & =  v_t \left( u_{s_{k}} \left( \oplus_{i=1}^k p_i \right) u_{s_{k}}^* \right) v_t^* \\
& =  v_t  \left( \oplus_{i=1}^{k} q_i \right)  v_t^* \\
&= \oplus_{i=1}^k q_i
\end{align}
as required by (\ref{fix}).  This completes our extension of $u_t$ to $[0,s_{k+1}]$.  Beginning at $t=0$ and continuing these extensions inductively yields a continuous path $u_t:[0,1) \to \mathcal{U}(A)$ such that
\begin{equation}\label{conjequal2}
u_t \left( \oplus_{i=1}^l p_i \right) u_t^* = \oplus_{i=1}^l q_i, \ \forall t \geq s_k, \ \forall l \leq k.
\end{equation}

Define $h:[0,1] \to A$ as follows: 
\begin{equation}\label{hdef}
h(t) = \left\{ \begin{array}{rl} u_t a u_t^*, &  t \in [0,1) \\ b, &t = 1 \end{array} \right. 
\end{equation}
Clearly, $\langle h(t) \rangle= \langle a \rangle$ in $\mathrm{Cu}(A)$ for every $t$ in $[0,1]$, and $h(t)$ is continuous at every $t$ in $[0,1)$ since $u_t$ is.  It remains to prove that $h(t)$ is continuous at $t=1$.  Recall that $s_k \to 1$ as $k \to \infty$.  For $t \geq s_k$ we have 
\begin{align}
\| h(1)-h(t) \| &= \| b - u_t a u_t* \| \\
&=\left\| \bigoplus_{i=1}^\infty \frac{1}{2^i} q_i - u_t \left( \bigoplus_{i=1}^\infty \frac{1}{2^i} p_i \right) u_t^* \right\| \\
& \stackrel{(\ref{conjequal2})}{=} \left\| \bigoplus_{i=k+1}^\infty \frac{1}{2^i} q_i - u_t \left( \bigoplus_{i=k+1}^\infty \frac{1}{2^i} p_i \right) u_t^* \right\| \\
& \leq \left\| \bigoplus_{i=k+1}^\infty  \frac{1}{2^i} q_i \right\| + \left\| \bigoplus_{i=k+1}^\infty \frac{1}{2^i} p_i  \right\| \\
&= \frac{1}{2^{k+1}} + \frac{1}{2^{k+1}} \\
&= \frac{1}{2^k}
\end{align}
This shows that $h(t)$ is continuous at $t=1$, completing the proof.
\end{proof}


\begin{thebibliography}{100}


\bibitem{APRT} Antoine, Ramon; Perera, Francesc; Robert, Leonel; Thiel, Hannes, \mbox{ C$^*$-algebras} of stable rank one and their Cuntz semigroups. {\it Duke Math. J.} 171 (2022), no. 1, 33–99. 


\bibitem{APT}  Ara, Pere; Perera, Francesc; Toms, Andrew S. K-theory for operator algebras. Classification of C$^*$-algebras. {\it Aspects of operator algebras and applications}, 1–71, Contemp. {\it Math., 534, Amer. Math. Soc., Providence, RI}, 2011.

\bibitem{BKR} Blackadar, Bruce; Kumjian, Alexander; Rørdam, Mikael, Approximately central matrix units and the structure of noncommutative tori. {\it K-Theory} 6 (1992), no. 3, 267–284.

\bibitem{BPT}  Brown, Nathanial P.; Perera, Francesc; Toms, Andrew S., The Cuntz semigroup, the Elliott conjecture, and dimension functions on \mbox{C$^*$-algebras.} {\it J. Reine Angew. Math.} 621 (2008), 191–211. 

\bibitem{CETWW} Castillejos, Jorge; Evington, Samuel; Tikuisis, Aaron; White, Stuart; Winter, Wilhelm, Nuclear dimension of simple C$^*$-algebras. {\it Invent. Math.} 224 (2021), no. 1, 245–290.

\bibitem{CEI} Coward, Kristofer T.; Elliott, George A.; Ivanescu, Cristian, The Cuntz semigroup as an invariant for C$^*$-algebras. {\it J. Reine Angew. Math.} 623 (2008), 161–193. 

\bibitem{Cuntz1} Cuntz, Joachim, Dimension functions on simple C$^*$-algebras. {\it Math. Ann.} 233 (1978), no. 2, 145–153. 

\bibitem{ET}  Elliott, George A.; Toms, Andrew S., Regularity properties in the classification program for separable amenable C$^*$-algebras. {\it Bull. Amer. Math. Soc. (N.S.)} 45 (2008), no. 2, 229–245.

\bibitem{KR} Kirchberg, Eberhard; R\o rdam, Mikael: Non-simple purely infinite C$^*$-algebras. {\it Amer. J. Math.} 122 (2000), no. 3, 637–666.

\bibitem{Ro1}  Rørdam, Mikael, The stable and the real rank of $\mathcal{Z}$-absorbing \mbox{ C$^*$-algebras.} {\it Internat. J. Math.} 15 (2004), no. 10, 1065–1084.

\bibitem{To1} Toms, Andrew S., On the classification problem for nuclear \mbox{ C$^*$-algebras.} {\it Ann. of Math. (2)} 167 (2008), no. 3, 1029–1044.


\bibitem{TW2}  Toms, Andrew S.; Winter, Wilhelm $\mathcal{Z}$-stable ASH algebras. {\it Canad. J. Math.} 60 (2008), no. 3, 703–720.

\bibitem{W1}  Winter, Wilhelm, Nuclear dimension and $\mathcal{Z}$-stability of pure \mbox{C$^*$-algebras.} {\it Invent. Math.} 187 (2012), no. 2, 259–342.

\end{thebibliography}
\end{document}